\documentclass[11pt]{article}

\usepackage[T1]{fontenc}
\usepackage{tikz}
\usepackage{graphicx}
\usetikzlibrary{arrows.meta,calc,decorations.pathreplacing}
\usepackage{graphicx,float,hyperref} 
\usepackage{amsmath,bbm,amsthm,amssymb,amsfonts,geometry,mathtools,enumerate}
\usepackage[mathscr]{euscript}
\usepackage{geometry} 
 \let\mathscr\relax
\usepackage[scr]{rsfso}
\newtheorem{theorem}{Theorem}[section]
\newtheorem{lemma}[theorem]{Lemma}

\newtheorem{corollary}[theorem]{Corollary}

\newtheorem{definition}[theorem]{Definition}

\usepackage{xfrac}

\newcommand{\N}{\mathbb{N}}
\newcommand{\R}{\mathbb{R}}
\newcommand{\Q}{\mathbb{Q}}
\newcommand{\Z}{\mathbb{Z}}

\newcommand{\inv}{^{-1}}

\usepackage{lipsum}
\usepackage{listings}
\usepackage{color}

\DeclareMathOperator{\Diff}{Diff}

\newcommand\restr[2]{{
  \left.\kern-\nulldelimiterspace 
  #1 
  \littletaller 
  \right|_{#2} 
  }}

\begin{document}

\title{The Rotation Number for a Generic $C^1$ Family of Circle Diffeomorphisms is Not Hölder}

\author{
  Xuzheng Lang
      \thanks{Department of Mathematics, University of California, Irvine. Email: xuzhenl@uci.edu.\\
      This work was partially supported by  NSF grant DMS$-$2247966 (PI: Anton Gorodetski)}
}

\date{\today}

\maketitle

\begin{abstract}
We study the optimal regularity of the rotation number as a function of the parameter in monotone one-parameter families of circle diffeomorphisms. We prove that for an open and dense set of $C^2$ monotone families, the optimal Hölder exponent is exactly $1/2$, showing therefore that the earlier result by J. Graczyk is sharp. For a $C^1-$ dense set of $C^{1+\alpha}$ families, we show that the rotation number cannot be more regular than $\alpha/(1+\alpha)-$Hölder. As a consequence, we get that the rotation number for a generic $C^1$ family of circle diffeomorphisms is not Hölder.
\end{abstract}

\section{Introduction}

The study of circle homeomorphisms and diffeomorphisms is a foundational topic in the theory of dynamical systems, tracing its origins to Henri Poincaré's work in the late 19th century \cite{poincare1885courbes}. To understand the asymptotic behavior of orbits, Poincaré introduced the notion of the rotation number, a topological invariant that measures the average rate at which points are translated around the circle. Let us start with the formal definition. 

Suppose that $f:S^1\to S^1$ is an orientation$-$preserving homeomorphism, and let $F:\R\to \R$ be any lift of $f.$ Then one can show that 
$$\rho(F): = \lim_{|n|\to \infty}\frac{F^n(x)-x}{n}$$
exists, and is equal for all $x\in \R.$ Moreover, if $F_1,F_2$ are any two lifts of $f,$ then $\rho(F_1)-\rho(F_2)\in \Z.$ 
This fact justifies the following definition:
\begin{definition}
    For a circle homeomorphism $f:S^1\to S^1,$ its rotation number is defined as
    $$\rho(f): = \{\rho(F)\},$$
    where $\{\cdot\}$ denotes the fractional part and $F$ is any lift of $f.$
\end{definition}

A recurring and central theme in the history of circle dynamics is the impact that the regularity of a map has on its dynamical properties. An example of this theme is the following celebrated theorem, proved by Arnaud Denjoy \cite{denjoy1932courbes} in 1932: any orientation$-$preserving $C^{2}$ circle diffeomorphism with an irrational rotation number is topologically conjugate to a rigid irrational rotation. Crucially, Denjoy, in the same work, also constructed $C^{1}$ counterexamples (in fact, Herman \cite{herman1979conjugaison} later showed that $C^{1+\alpha}$ counterexamples exist for any $\alpha<1$)$-$diffeomorphisms with irrational rotation numbers that exhibit wandering intervals and thus cannot be conjugate to rigid rotations. Denjoy theory shows that sufficient regularity of the circle diffeomorphism force topological conjugacy to a rotation in the irrational case. Herman \cite{herman1979conjugaison} and Yoccoz \cite{Yoccoz1984} then developed a finer rigidity theory, showing that under arithmetic conditions on the rotation number the conjugacy inherits high regularity. In this direction, Khanin and Sinai \cite{SinaiKhanin1989} studied the smoothness of conjugacies between circle diffeomorphisms and rotations, and Khanin–Teplinsky later obtained sharp low-regularity versions of Herman’s theorem. See \cite{KhaninTeplinsky2009} for more detial.

These results concern the regularity of the conjugacy for a fixed map, assuming both smoothness of the diffeomorphism and arithmetic information about its rotation number. The present paper considers a complementary problem: instead of fixing a single diffeomorphism, we study families of diffeomorphisms and ask how the regularity of the family controls the regularity and structure of the function $t\mapsto\rho(f_t).$

More specifically, we will consider families $\{f_t\}_{t\in [0,1]}$ where each $f_t:S^1\to S^1$ is an orientation$-$preserving homeomorphism, such that there exists a lift $\{F_t\}_t$ of the family with $t\mapsto F_t$ being continuous and $F_{t_1}\leq F_{t_2}$ whenever $t_1\leq t_2.$ It is well known that for such a family the rotation number $\rho(t): = \rho(f_t)$ is, under suitable conditions, monotone and continuous, and behaves like a ``devil's staircase"$-$where the stability of rational rotation numbers under small perturbations create robust intervals in the parameter space known as "mode locking." See chapter 11 in \cite{katok_hasselblatt_1995} for a modern exposition. 

With these properties in mind, it is natural to investigate the relationship between the regularity of the family and the regularity of the rotation number. In this paper, we are concerned with families that are $C^2,\;C^1,$ and $C^{1+\alpha}:$
\begin{definition}
\label{def:C^k families}
    Let $k\in \N$. We say that a family of orientation$-$preserving circle maps $\{f_t\}_t$ is a $C^k$ family if there exists a lift $\{F_t\}_t$ of the family such that the function $F(t,x): = F_t(x)$ is $C^k.$ And for $\alpha\in (0,1),$ we say that a family is $C^{k+\alpha}$ if $F(t,x)$ is $C^{k+\alpha}.$

    Furthermore, we say that the family is monotone if $\frac{\partial F(x,t)}{\partial t}\geq 0,$ and strictly monotone if $\frac{\partial F(x,t)}{\partial t}> 0$.

    Throughout this paper, we will denote by $R_t$ the rigid circle rotation by $t,$ and by $\Diff_+^k(S^1)$ and $\Diff_+^{k+\alpha}(S^1)$ the spaces of orientation-preserving $C^k$ and $C^{k+\alpha}$ circle diffeomorphisms, respectively.
\end{definition}
The main results are the following:
\begin{enumerate}
    \item For a generic $f\in \Diff^2_+(S^1),$ the rotation number of the family $\{f_t\}_t = \{R_t\circ f\}_{t\in [0,1]}$ is exactly $1/2-$Hölder.
    \item For a generic monotone $C^2$ family, the rotation number of the family is exactly $1/2-$Hölder.
    \item For any $\alpha\in (0,1)$, the set of families whose rotation number is at most $\frac{\alpha}{1+\alpha}-$Hölder is $C^1-$dense in the space of monotone $C^{1+\alpha}$ families.
    \item For a generic monotone $C^1$ family, the rotation number of the family is not Hölder continuous.
\end{enumerate}

We will provide the formal statements of the results at the end of this section.

Some work has already been done to explore this relationship between the regularity of the family and that of the rotation number. In 1991, Graczyk \cite{graczyk1991harmonic} considered strictly monotone $C^2$ families of circle diffeomorphisms, and showed that in this case the rotation number depends Hölder continuously on the parameter value, with Hölder exponent at least $1/2.$ A similar regularity result extends to specific families of critical circle maps. In 1996, Graczyk and Światek \cite{graczyk1996critical} proved that when the family $\{f_t\}_t$ has the form $f_t = R_t\circ f_0,$ where $f_0$ is any $C^k$ ($k\geq 3$) circle homeomorphisms such that the derivative vanishes exactly at $0$ and $f^{(\ell)}(0)\neq 0$ for some $\ell,$ then the rotation number is a Hölder continuous function of $t.$ In the case when $f_0$ is a $C^\infty$ circle diffeomorphism, Matsumoto \cite{matsumoto2012derivatives} in 2012 proved that the family $\{R_t\circ f_0\}_t$ will have $\limsup_{t'\to t}(\rho(t')-\rho(t))/(t'-t)\geq 1$ at any $t$ at which $\rho(t)$ is irrational. As for parameter values at which the rotation number is rational, full differentiability can be shown under some rather restrictive conditions. Parkhe \cite{parkhe2013one} in 2013 proved that if $\{f_t\}_t$ is a $C^2$ family of circle homeomorphisms such that the rotation number $\rho(t)$ is strictly increasing in $t,$ then $\rho$ is differentiable at any $t_0$ where $\rho(t_0)\in \Q.$

The regularity of the rotation number is also closely related to questions in spectral theory. For a one-dimensional ergodic Schrödinger operator, the integrated density of states can be expressed in terms of the rotation number of the corresponding Schrödinger cocycle \cite{delyon1983rotation}, with the energy playing the role of the parameter. In this way, regularity estimates for the rotation number may be viewed as dynamical analogues of regularity estimates for the integrated density of states. This perspective was recently emphasized by Gorodetski and Kleptsyn \cite{gorodetski_kleptsyn_log_holder}, who, in 2025, obtained a dynamical proof of the 1D Craig–Simon theorem \cite{craig1983log}, \cite{craig1983subharmonicity} on log-Hölder continuity of the integrated density of states, by showing the log-Hölder continuity of the rotation number  for families of smooth circle cocycles over an ergodic base. See also \cite{Duarteetal2025} for more results on the fibered rotation number for families of circle cocycles, and \cite{EmilsdottirMonakov2026} for results on the rotation number of families of random circle diffeomorphisms and its relation to the Anderson model. Although the setting of the present paper is simpler, this connection provides additional motivation for understanding the sharp dependence of the rotation number on the regularity of the underlying family.

Returning to our discussion on the regularity of the rotation number, Gorodetski and Kleptsyn's result \cite{gorodetski_kleptsyn_log_holder} on smooth cocycles implies that if we have a $C^1$ family of circle diffeomorphisms, then the rotation number is log-Hölder continuous. It is therefore natural and interesting to ask what happens to the rotation number when our family of diffeomorphisms has regularity between $C^1$ and $C^2:$ on the one hand we know it must be at least log-Hölder, and on the other hand Graczyk's work \cite{graczyk1991harmonic} tells us it could be perhaps Hölder-1/2. But is $1/2-$Hölder the best we can do for $C^2$ families? And for a generic $C^1$ family, will Hölder continuity somehow survive? What happens with $C^{1+\alpha}$ families?

We will partially answer these questions with our main results. To state them, we need the following notations:
\begin{definition}
    For $k\in \N$, denote by $\Lambda_k$ the space of all one-parameter monotone $C^k$ families $\{f_t\}_t:[0,1]\to \Diff_+^k(S^1)$ with non-constant rotation number, i.e., $\rho(f_1)>\rho(f_0).$

    Similarly, for $\alpha\in (0,1),$ denote by $\Lambda_{k+\alpha}$ the space of all one-parameter monotone $C^{k+\alpha}$ families $\{f_t\}_t:[0,1]\to \Diff_+^{k+\alpha}(S^1)$ with non-constant rotation number.
\end{definition}

Now we are ready to formulate the {\bf statements of the main results:}

\begin{theorem}
    \label{thm:R_t f 1/2-Holder}
    Consider the space $\Diff_+^2(S^1)$ of all orientation$-$preserving $C^2$ diffeomorphisms of the circle, equipped with the $C^2$ topology. Then there is an open dense subset $E\subseteq\Diff_+^2(S^1)$ such that for all $f\in E,$ the rotation number $\rho(t): = \rho(R_t\circ f)$ is Hölder continuous with exponent exactly $1/2$ (and not with any exponent greater than $1/2$).
\end{theorem}
Theorem~\ref{thm:R_t f 1/2-Holder} is, in some sense, a special case of Theorem~\ref{thm:generic C^2 1/2-Holder}, which deals with general monotone $C^2$ families and provides an affirmative answer to our question as to whether the Hölder exponent $1/2$ shown in Graczyk's work \cite{graczyk1991harmonic} is generically optimal.
\begin{theorem}
    \label{thm:generic C^2 1/2-Holder}
    Consider $\Lambda_2$ equipped with the metric $d_2(\{f_t\}_t,\{g_t\}_t): = \|F(t,x)-G(t,x)\|_{C^2}$. Then there exists an open and dense subset $E\subseteq \Lambda_2$ such that for all $\{f_t\}_t\in E,$ the rotation number $\rho(t): = \rho(f_t)$ is Hölder continuous with exponent exactly $1/2$ (and not with any exponent greater than $1/2$). 
\end{theorem}

Theorem~\ref{thm:generic C^1 not Holder} tells us, somewhat surprisingly, that Hölder continuity does not survive for a generic $C^1$ family.
\begin{theorem}
    \label{thm:generic C^1 not Holder}
    Consider $\Lambda_1$ equipped with the metric $d_1(\{f_t\}_t,\{g_t\}_t): = \|F(t,x)-G(t,x)\|_{C^1}$.  Then the set of families in $\Lambda_1$ for which the rotation number is not Hölder continuous is dense $G_\delta.$
\end{theorem}
And finally, Theorem~\ref{thm:alpha/(1+alpha) Holder in C^(1+alpha) dense} gives us some idea about what could happen with $C^{1+\alpha}$ families.

\begin{theorem}
    \label{thm:alpha/(1+alpha) Holder in C^(1+alpha) dense}
    Let $\alpha\in (0,1).$ Then the set of families in $\Lambda_{1+\alpha}$ for which the rotation number is not $\beta-$Hölder for any $\beta>\frac{\alpha}{1+\alpha}$ is dense under the $d_1$ metric.
\end{theorem}


\paragraph{Structure of the paper.}
    In section 2, we prove Theorem~\ref{thm:R_t f 1/2-Holder} and develop the key bottleneck estimation in Lemma~\ref{lem:bottleneck} which we will use multiple times in later sections. In section 3 we prove Theorem~\ref{thm:generic C^2 1/2-Holder} where the analysis remains more or less the same as in Theorem~\ref{thm:R_t f 1/2-Holder}, but put in a more general setting. In section 4, we first prove Theorem~\ref{thm:alpha/(1+alpha) Holder in C^(1+alpha) dense}, which shares the same key ideas as in the proof of Theorem~\ref{thm:R_t f 1/2-Holder}, and then Theorem~\ref{thm:generic C^1 not Holder}, which follows more or less immediately via an application of the Baire Category Theorem.

\section{Generic Optimality of $1/2-$Hölder for $\{R_t\circ f\}_t$}
This section will be dedicated to the proof of Theorem~\ref{thm:R_t f 1/2-Holder}: we will consider the space $\Diff_+^2(S^1)$ equipped with the $C^2$ topology, and construct an open and dense subset $E\subseteq\Diff_+^2(S^1)$ such that for all $f\in E,$ the rotation number of the family $\rho(t): = \rho(R_t\circ f)$ is exactly $1/2-$Hölder. 

Let us first recall the following result from Graczyk:
\begin{theorem}\cite[Corollary 1]{graczyk1991harmonic}
\label{thm:Holder 1/2} 
    Suppose that $\{f_t\}_t\subseteq \Diff_+^2 (S^1)$ for $t\in [0,1]$ is a strictly monotone $C^2$ family of orientation$-$preserving circle diffeomorphisms. Then the rotation number $\rho(t): = \rho(f_t)$ is $1/2-$Hölder continuous. 
\end{theorem}
This means our work in this sections reduces to finding an open and dense subset $E\subseteq \Diff_+^2(S^1)$ such that the families $\{R_t\circ f\}_{t\in [0,1]}$ with $f\in E$ cannot have rotation number with regularity better than $1/2-$Hölder. To explain our strategy, we need the following definition:
\begin{definition}
    Let $f$ be an orientation$-$preserving circle homeomorphism with $\rho(f) = p/q$. Then we say that $f$ has a parabolic periodic point $x_0$ if $(f^q)'(x_0) = 1.$ 

    Furthermore, we say that such a periodic point is quadratic if $(f^q)''(x_0)\neq0.$
\end{definition}
Here is the idea of the proof: since at rational rotation numbers $\rho$ will exhibit mode locking, i.e., the existence of an interval in the parameter space on which $\rho$ is constant, we can consider the parameter $t_0\in (0,1)$ that is at the upper edge of the mode locking plateau, i.e., $t_0 = \sup\rho\inv(p/q).$ Then we are done if we can show that $\rho(t_0+\varepsilon)-\rho(t_0)\geq c\varepsilon^{1/2}.$

To this end, let us consider $f_{t_0}^q.$ Note that $f_{t_0}^q-x$ must have zeros, i.e, the graph of $f_{t_0}^q$ must have tangency with identity. Now let us assume for a moment that there are finitely many such tangencies, and they are all quadratic parabolic. If we zoom in on one of the tangencies and consider the graph of $f_{t_0+\varepsilon}^q,$ we see that it will look like the graph of $f_{t_0}^q,$ but shifted upward by approximately $\varepsilon,$ as shown in Figure~\ref{fig:quadratic-bottleneck}. And so there is a gap between $f_{t_0+\varepsilon}^q$ and identity: all fixed points of $f_{t_0}^q$ are destroyed and therefore the rotation number strictly increases. To estimate this increase, it is enough to estimate instead the number of steps for a point to travel once around the circle, for the increase in $\rho$ is inversely proportional to this quantity. But notice that the point will only slow down at the ``bottlenecks" near the fixed points of $f_{t_0}^q,$ and so total number of steps needed to cover the whole circle is on the order of the number of steps needed to clear one of the bottlenecks. 

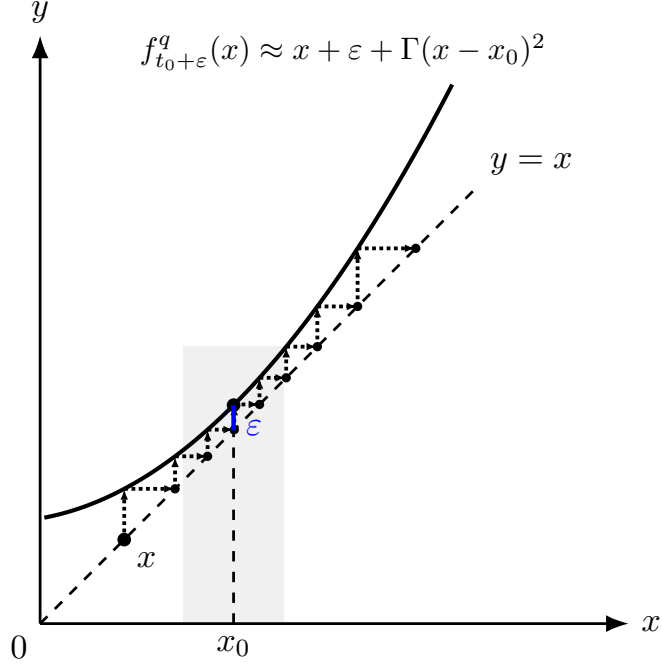
\begin{figure}[htbp]
\centering

\resizebox{0.6\textwidth}{!}{%
\begin{tikzpicture}[x=0.95cm,y=0.95cm,>=Latex,thick,scale=0.9,transform shape]

\begin{scope}

  \def\xzero{2.3}
  \def\eps{0.3}
  \def\gam{0.18}

  \draw[->] (0,0) -- (7.0,0) node[right] {$x$};
  \draw[->] (0,0) -- (0,7.0) node[above] {$y$};
  \node[below left] at (0,0) {$0$};

  \fill[black!6] (\xzero-0.6,0.02) rectangle (\xzero+0.6,3.3);

  \draw[dashed] (0,0) -- (5.20,5.20) node[above right] {$y=x$};

  \draw[black,very thick,domain=0.05:4.9,samples=200]
    plot (\x, {\x + \eps + \gam*(\x-\xzero)^2});

  \pgfmathsetmacro{\xcur}{1.0}

  \fill[black] (\xcur,\xcur) circle (2.2pt);
  \node[black,below right] at (\xcur,\xcur) {$x$};

  \foreach \k in {0,...,7}{
    \pgfmathsetmacro{\ynext}
      {\xcur + \eps + \gam*(\xcur-\xzero)*(\xcur-\xzero)}

    \draw[
      black,
      densely dotted,
      line width=1.1pt,
      -{Latex[length=1.1mm,width=0.8mm]}
    ]
      (\xcur,\xcur) -- (\xcur,\ynext);

    \draw[
      black,
      densely dotted,
      line width=1.1pt,
      -{Latex[length=1.1mm,width=0.8mm]}
    ]
      (\xcur,\ynext) -- (\ynext,\ynext);

    \fill[black] (\ynext,\ynext) circle (1.5pt);

    \xdef\xcur{\ynext}
  }

  \node[black,font=\small] at (3.6,6.8)
    {$f_{t_0+\varepsilon}^q(x)\approx x+\varepsilon+\Gamma(x-x_0)^2$};

  \coordinate (P) at (\xzero,{\xzero+\eps});
  \coordinate (Q) at (\xzero,\xzero);

  \fill[black] (P) circle (2.2pt);

  \draw[dashed] (\xzero,0) -- (P);
  \node[below] at (\xzero,0) {$x_0$};

  \draw[blue,line width=1.3pt] (Q) -- (P);
\node[blue,right] at ($(P)!0.9!(Q)$) {$\varepsilon$};

\end{scope}

\end{tikzpicture}
}

\caption{Passage through a quadratic bottleneck (grey region) under iteration of 
$f_{t_0+\varepsilon}^q$.}
\label{fig:quadratic-bottleneck}
\end{figure}

Here is the heuristics behind the exponent $1/2$: since this bottleneck is quadratic, for any point $x$ near some fixed point $x_0$ of $f_{t_0}^q$ the time-one-shift vector field $\dot x = \varepsilon+\gamma(x-x_0)^2$ well approximates the movement of $x$ under the iterations of $f_{t_0+\varepsilon}^q.$ Therefore, simply by integrating, we see that the time needed for $x$ to traverse the region $[x_0-\delta,x_0+\delta]$ is
$$\int_{x_0-\delta}^{x_0+\delta}\frac{dx}{\varepsilon+\gamma(x-x_0)^2} = \frac{2}{\sqrt{\varepsilon\gamma}}\arctan \left(\delta\sqrt{\frac{\gamma}{\varepsilon}}\right) \asymp 1/\sqrt{\varepsilon}.$$
So the increase in $\rho$ must be at least $\sqrt{\varepsilon},$ hence $\rho(t)$ is at most $1/2-$Hölder.

These estimates will also be used in both section 3 and section 4, and will be made rigorous in the following lemma. Since the $C^{1+\alpha}$ case is an exact analogue to the $C^2$ case just described, we will prove a generalized version that works for them both. While direct comparison to a vector field yields the correct order of magnitude and is a good piece of intuition to keep in mind, we will not use it in the proof.

\begin{lemma}
\label{lem:bottleneck} 
    Let $\{f_t\}_{t\in [0,1]}\subseteq \Diff_+^1 (S^1)$ be a strictly monotone $C^1$ family. Let $\alpha\in (0,1]$. Suppose that there exists some $t_0\in (0,1)$ such that $\rho(f_{t_0}) = p/q$, and $f_{t_0}$ has only finitely many periodic points. Let $T$ be the unique lift of $f_{t_0}^q$ with fixed points.
    Suppose further that for each periodic point $x_k$ and any $x$ sufficiently close $x_k$,
    one has
    $$T(x)\geq x+c|x-x_k|^{1+\alpha}$$
    for some constant $c>0$. Then $\rho(t): = \rho(f_t)$ cannot be Hölder continuous for any exponent $\beta>\frac{\alpha}{1+\alpha}.$
\end{lemma}
\begin{proof}
    We know $T(x)\geq x+c|x-x_0|^{1+\alpha}$ for all $x$ near $x_0,$ where $x_0$ is a fixed point of $T.$ From the definition of the rotation number we can see that that if we fix $x\in \R$ and let $N(F,x)$ be the least integer satisfying $F^{N(F,x)}(x)\geq x+1,$ then
    $\rho(f) = O(1/N(F,x))$ and $1/N(F,x) = O(\rho(f)).$ 
    So it suffices to estimate the number of steps it takes for $x$ to make it around the circle under the action of $f_{t_0+\varepsilon}.$ 
    
  Observe that we can make further reductions: since our family is $C^1$ and strictly monotone, we have $0<c<\frac{\delta}{\delta t}F_t^q(x)$ for some constant $c$, and so $F_{t_0}^q(x)+c\varepsilon\leq F_{t_0+\varepsilon}^q(x)$ everywhere. This means $N(F_{t_0}^q+c\varepsilon,x)\geq N(F_{t_0+\varepsilon}^q,x).$ So if we can show that $N(F_{t_0}^q+c\varepsilon,x) =  O(\varepsilon^{-\alpha/(1+\alpha)}),$ we would have $\rho(F_{t_0+\varepsilon}^q)\geq \rho(F_{t_0}^q+c\varepsilon) = C\varepsilon^{\alpha/(1+\alpha)}$ and $\rho(f_{t_0+\varepsilon})-\rho(f_{t_0}) \geq C\varepsilon^{\alpha/(1+\alpha)}$ for some constant $C>0,$ as desired. Without loss of generality we can assume $c = 1.$
    
    Let us fix some sufficiently small constant $\delta>0,$ and suppose without loss of generality that $x_0 = 0$. Now we estimate the number of steps it will take for some point $x$ to cross $[-\delta,\delta]$ under the iterations of $G: = T+\varepsilon.$ 

    To this end, let us write $r_\varepsilon: = \varepsilon^{1/(1+\alpha)},$ and divide the interval $[-\delta,\delta]$ into three regions: $[-\delta,-r_\varepsilon]\cup[-r_\varepsilon,r_\varepsilon]\cup[r_\varepsilon,\delta].$ If $|x|\leq r_\varepsilon,$ then we can use the fact that $G(x)-x\geq \varepsilon$ to conclude that the number of steps needed to cover the middle region is at most $O(\varepsilon^{1/(1+\alpha)-1}) = O(\varepsilon^{-\alpha/(1+\alpha)}).$

    For $x\in [r_\varepsilon,\delta],$ let us divide this interval into sections of the form 
    $$\{[2^kr_\varepsilon,2^{k+1}r_\varepsilon]\}_{k = 0}^K,$$
    where $K\in \N$ satisfies $2^Kr_\varepsilon\leq \delta<2^{K+1}r_\varepsilon.$ If $x\in [2^kr_\varepsilon,2^{k+1}r_\varepsilon],$ then by assumption the minimum step size is
    $$G(x)-x\geq |x|^{1+\alpha}\geq (2^kr_\varepsilon)^{1+\alpha}.$$
    And so the number of steps needed to cover this section is at most 
    $$\frac{2^kr_\varepsilon}{(2^kr_\varepsilon)^{1+\alpha}}+1 = (2^kr_\varepsilon)^{-\alpha}+1.$$
    So the total number of steps needed to cover the whole interval $[r_\varepsilon,\delta]$ is at most
    $$\sum_{k = 0}^K2^{-\alpha k}r_\varepsilon^{-\alpha}+(K+1)\leq O(r_\varepsilon^{-\alpha}) = O(\varepsilon^{-\alpha/(1+\alpha)}),$$
    since $K$ is only of order $\log\varepsilon\inv.$

    For $x\in [-\delta,-r_\varepsilon],$ we again divide the interval into sections of the form $$\{[-2^{k+1}r_\varepsilon,-2^{k}r_\varepsilon]\}_{k = 0}^K,$$ 
    and again the minimum step size for $x\in [-2^{k+1}r_\varepsilon,-2^{k}r_\varepsilon]$ is
    $$G(x)-x\geq |x|^{1+\alpha}\geq (2^kr_\varepsilon)^{1+\alpha},$$
    and so going through the calculations as above will produce the same estimate. 

    Finally, since there are finitely many such bottlenecks and it takes only $O(1)$ many steps for $g$ to clear the regions outside of the bottlenecks, we have, in total, $N(g,x) = O(\varepsilon^{-\alpha/(1+\alpha)})$ for any $x\in S^1.$ This completes the proof.
    
\end{proof}

Now we see that as long as our family $\{f_t\}_t$ is ``good'' in the sense that it contains some $f_{t_0}$ with only quadratic parabolic periodic points, then the rotation number of this family can be no better than $1/2-$Hölder. So in order to prove our target theorem, all it remains to show is that those diffeomorphisms $f\in \Diff_+^2(S^1)$ for which $\{R_t\circ f\}_t$ is a ``good'' family, is open and dense:
\begin{lemma}
\label{lem:R_tf: openess of quadratic tangency} 
    If for some $f\in \Diff_+^2(S^1)$ and some $t\in (0,1)$ $f_t$ has periodic points which are all parabolic and quadratic, then there are only finitely many of them. Further more, any $g\in \Diff_+^2(S^1)$ sufficiently $C^2$-close to $f$ will have some $t'\in (0,1)$ such that $g_{t'}$ has periodic points which are all parabolic and quadratic.
\end{lemma}
\begin{proof}
    Take $f\in \Diff_+^2(S^1)$ and $t\in (0,1)$ such that $\rho(t) := \rho(f_t) = p/q$ and $f_t$ has only parabolic and quadratic periodic points. First notice that the number of periodic points of $f_t$ must be finite. Indeed, if not, then there is a sequence $(x_n)_n$ of fixed points of $f_t^q$ that converge to some $x_0\in S^1.$ Then $x_0$ is a fixed point of $f_t^q$ with $(f_t^q)'(x_0) = 1$ by continuity. However, 
    $$(f_t^q)''(x_0) = \lim_n\frac{(f_t^q)'(x_n)-(f_t^q)'(x_0)}{x_n-x_0} = 0,$$
    contradicting our assumption that every periodic point of $f_t$ is parabolic and quadratic. 
    
    Since the mode locking interval $I_f(p/q): = \rho\inv(p/q)$ is non-degenerate, we can take $g\in \Diff_+^2(S^1)$ that is $C^2-$close to $f$ and a small $\delta>0$ such that $t-\delta\in int(I_f(p/q)),$ and
    $$\rho(g_{t-\delta}) = \rho(f_{t-\delta}) = p/q.$$
    Since $\rho(f_{t+\delta})>\rho(f_t)$ for any $\delta>0,$ we can ensure that $$\rho(g_{t+\delta})>\rho(g_t) = p/q$$ provided $g$ is sufficiently close to $f.$ 
    Therefore, there exists $t'\in [t-\delta,t+\delta]$ such that $\rho(g_{t'}) = p/q$ and $g_{t'}$ has only parabolic periodic orbits. We just need to show that these are also quadratic. 

    By assumption, if $x\in S^1$ is a fixed point of $f_t^q$, then $|(f_t^q)''(y)|>0$ for any $y$ in a small neighborhood of $x.$ Since $g_{t'}^q$ and $f_t^q$ are $C^2-$close, all fixed points of $g_{t'}^q$ will be sufficiently close to those of $f_t^q,$ and so $|(g_{t'}^q)''|$ will be also bounded away from 0 at its fixed points. This completes the proof.
\end{proof}

    The above lemma shows that the collection of $f\in \Diff_+^2(S^1)$ such that $\{f_t\}_t$ make a ``good'' family, is open, and the following lemma tells us that it is also dense.

\begin{lemma}
\label{lem:R_tf: denseness of quadratic tangency} 
    If $f\in \Diff_+^2(S^1)$ has fixed points which are all parabolic, then there is some $h\in \Diff_+^2(S^1)$ arbitrarily $C^2-$close to $f$ such that $h$ has fixed points which are all parabolic and quadratic.
\end{lemma}
\begin{proof}
    Suppose that $f$ has at least one parabolic fixed point $x_0\in S^1$ that is non-quadratic, i.e., $f''(x_0) = 0.$ Let $F$ be the unique lift of $f$ with fixed points. Without loss of generality $F(x)\geq x$ for all $x\in\R$. Now take any $g\in \Diff_+^2(S^1)$ such that:
    \begin{enumerate}
        \item $x_0$ is the only fixed point of $g$;
        \item $x_0$ is a parabolic quadratic fixed point of $g$;
        \item for the lift $G$ of $g$ with fixed point $x_0,$ $G(x)>x$ for any $x\in \R$ that is not a lift of $x_0.$
        \item $\|G(x)-x\|_{C^2}$ is sufficiently small. 
    \end{enumerate}
    Now consider $h: = g\circ f$, which can be made arbitrarily $C^2$ close to $f$. In this case, $x_0$ is the only fixed point of $h$. Finally, we can verify by a straightforward calculation that 
    $$h(x_0) = x_0,\;h'(x_0) = 1,\;h''(x_0) = g''(x_0)\neq 0,$$
    as desired. 
\end{proof}
Now putting the pieces together finishes the proof:
\begin{lemma}
\label{lem:R_tf: open denseness of quadratic tangency} 
    Let $E\subseteq \Diff_+^2(S^1)$ be the set of all $f\in \Diff_+^2(S^1)$ for which there exists some $t\in [0,1]$ such that $f_t$ has rational rotation number and only parabolic quadratic periodic orbits. Then $E\subseteq \Diff_+^2(S^1)$ is open and dense in the $C^2$ topology.
\end{lemma}
\begin{proof}
    $E$ is open by Lemma~\ref{lem:R_tf: openess of quadratic tangency}. We now show it is dense. First notice that for any $f\in \Diff_+^2(S^1)$, the family $\{R_t\circ f\}_{t\in [0,1]}$ will have some $f_{t_0}$ with only parabolic periodic points. Indeed, since $\rho(f_1) = \rho(f_0)+1$, there must be some $t\in [0,1)$ such that $\rho(f_{t}) = 0$. Then for $t_0: = \sup\rho\inv(0)$, $f_{t_0}$ must have only parabolic periodic points.

    Now if $f\in E^c,$ then by Lemma~\ref{lem:R_tf: denseness of quadratic tangency} there is some $h\in \Diff_+^2(S^1)$ arbitrarily close to $f,$ such that $h_{t_0}$ has fixed points that are all parabolic and quadratic. This means $h\in E,$ so $E$ is dense. 
\end{proof}

    Combining Theorem~\ref{thm:Holder 1/2}, Lemma~\ref{lem:bottleneck}, and Lemma~\ref{lem:R_tf: open denseness of quadratic tangency}, we obtain a proof of Theorem~\ref{thm:R_t f 1/2-Holder}.

\section{Generic Optimality of $1/2-$Hölder for General $C^2$ families}

This section will be dedicated to the proof of Theorem~\ref{thm:generic C^2 1/2-Holder}: we will consider the space $\Lambda_2$ of all monotone $C^2$ families $\{f_t\}_t:[0,1]\to \Diff_+^2(S^1)$ with $\rho(t_1)>\rho(t_0).$ As a reminder, this means that for any fixed continuous lift of the family $\{F_t\}_t,$ the function $F(t,x): = F_t(x)$ is $C^2,$ and $\partial_t F(t,x)\geq 0$.
We also equip this space with the metric $d_2(\{f_t\}_t,\{g_t\}_t): = \|F(t,x)-G(t,x)\|_{C^2}$.
Under these assumptions, we construct an open and dense subset $\mathcal E\subseteq \Lambda_2$ such that for all $\{f_t\}_t\in \mathcal E,$ then rotation number of the family $\rho(t): = \rho(f_t)$ is exactly $1/2-$Hölder.

Just as in section 2, Graczyk's Theorem from \cite{graczyk1991harmonic} and Lemma~\ref{lem:bottleneck} tells us that to prove the desired result, it is enough to show that the collection of ``good'' families, i.e., families $\{f_t\}_t\in \Lambda_2$ that contain some $f_{t_0}$ with only quadratic parabolic periodic points, is open and dense in $\Lambda_2.$ The next lemma show that the ``good'' families are dense.

\begin{lemma}
\label{lem:denseness of tangency(general)}
consider the subset $\mathcal E\subseteq \Lambda_2$ consisting of those families $\{f_t\}_t$ for which there is some $t_0\in (0,1)$ such that $f_{t_0}$ has periodic points that are all quadratic and parabolic. Then $\mathcal E$ is dense in $\Lambda_2.$ 
\end{lemma}
\begin{proof}
Let us first consider $\mathcal F\subseteq \Lambda_2$ consisting of those families $\{f_t\}_t$ for which there is some $t_0\in (0,1)$ such that $f_{t_0}$ have periodic points that are all parabolic. We will first show that Then $\mathcal F$ is dense in $\Lambda_2.$

    It suffices to show that for any $\{f_t\}\in \Lambda_2\setminus \mathcal F,$ there is some $\{g_t\}\in \mathcal F$ arbitrarily close to $\{f_t\}$ in norm. If $\{f_t\}_t$ is not strictly increasing, then we can consider the family $\{f_t\circ R_{\varepsilon t}\}_t$ for arbitrarily small $\varepsilon>0.$ So we may assume without loss of generality our family $\{f_t\}_t$ is strictly increasing. 
    
    Now consider any $f_{t_0}$ with $\rho(f_{t_0}) = p/q\in (\rho(f_0),\rho(f_1))$, such that $t_0$ is on the upper boundary of the mode locking interval at $p/q$ (if the mode locking is a single point, let $t_0$ be that point), and suppose without loss of generality that $0$ is a fixed point of $f_{t_0}^q.$ Take any $h\in \Diff_+^2(S^1)$ with $\{f_{t_0}^k(0)\}_{k = 0}^{q-1}$ as the only fixed points, such that there exists a lift $H$ with $\|H(x)-x\|_{C^2}$ sufficiently small, and $H(x)\geq x$. Define $g_t: = h\circ f_t.$ Then $\{g_t\}_t$ is a strictly increasing family with parabolic periodic  points $\{f_{t_0}^k(0)\}_{k = 0}^{q-1}$. Since $d_2(\{g_t\}-\{f_t\})$ can be made as small as we wish, $\mathcal F$ is dense in $\Lambda_2.$

    Now let us suppose $\{f_t\}_t\in \mathcal F,$ and find some $\{g_t\}_t\in \mathcal E$ arbitrarily close to $\{f_t\}_t.$ Let $t_0\in (0,1)$ be such that $\rho(f_{t_0}) = p/q$ where all periodic points of $f_{t_0}$ are parabolic. Let $T$ be the unique lift of $f_{t_0}^q$ with fixed points. Then by the specific construction above, we may assume that $T(x)\geq x,$ and all periodic points of $f_{t_0}$ are of the form $\{f_{t_0}^k(x_0)\}_{k=0}^{q-1} = \{x_k\}_{k=0}^{q-1}.$

    For each $k,$ fix an open set $U_k$ containing every lift of $x_k,$ such that the $U_k's$ are pairwise disjoint. Take any $\phi\in C^\infty(\R)$ such that:
    \begin{enumerate}
        \item $\phi$ vanishes outside of $\bigcup_k^{q-1}U_k$.
        \item $\phi(x) =\phi (x+1)$ for all $x\in \R.$
        \item $\phi(x)\geq 0.$
        \item $\max(\|\phi\|_\infty,\; \|\phi'\|_\infty,\;\|\phi''\|_\infty)\leq 1.$
        \item $\phi(x_k) = \phi'(x_k) = 0,$ and $\phi''(x_k) = 1$ for all $k$.
    \end{enumerate}
    Now consider the $C^2$ family $\{g_t\}_t$ given by the lifts $G_t = F_t+\varepsilon\phi$ for some small $\varepsilon>0.$ Then $d_2(\{f_t\}_t-\{g_t\}_t)$ can be made arbitrarily small.

    We verify that each $g_t$ is indeed a circle diffeomorphism: we have $$G_t(x+1) = F_t(x+1)+\varepsilon\phi(x+1) = F_t(x)+\varepsilon\phi(x)+1 = G_t(x)+1.$$ Also, $G_t'(x) = F_t'(x)+\varepsilon\phi'(x)\geq F_t'-\varepsilon>0$ as long as $\varepsilon<\min_{t,x}F_t'.$ And so $g_t$ must be a $C^2$ diffeomorphism.

    Now it remains to show that $\{g_t\}_t$ belongs to $\mathcal E.$ Indeed, if $\varepsilon$ is sufficiently small, $\{x_k\}_k^{q-1}$ will be the only periodic points of $g_{t_0}.$ And at these points, we have
    $$(g_{t_0}^q)'(x_k) =\prod_{i=0}^{q-1}g_{t_0}'(x_i) = \prod_{i=0}^{q-1}f_{t_0}'(x_i) = (f_{t_0}^q)'(x_k) = 1,$$
    and
    $$(g_{t_0}^q)''(x_k) = \sum_{i=0}^{q-1}\frac{g_{t_0}''(x_i)}{g_{t_0}'(x_i)}(g_{t_0}^i)'(x_k).$$
    But $g_{t_0}'(x_i) = f_{t_0}'(x_i)$ and $g_{t_0}''(x_i) = f_{t_0}''(x_i)+\varepsilon.$ So we get 
    $$(g_{t_0}^q)''(x_k) = \sum_{i=0}^{q-1}\frac{f_{t_0}''(x_i)}{f_{t_0}'(x_i)}(f_{t_0}^i)'(x_k)+\varepsilon\left(\sum_{i=0}^{q-1}\frac{(f_{t_0}^i)'(x_k)}{f_{t_0}'(x_i)}\right) = (f_{t_0}^q)''(x_k)+\varepsilon\left(\sum_{i=0}^{q-1}\frac{(f_{t_0}^i)'(x_k)}{f_{t_0}'(x_i)}\right),$$
    which is strictly positive, as desired. 
\end{proof}
Finally, the next lemma tells us that the ``good'' families are open in $\Lambda_2.$
\begin{lemma}
    \label{lem:openess of quadratic tangency(general)}
     The subset $\mathcal E$ defined as in Lemma~\ref{lem:denseness of tangency(general)} is open in $\Lambda_2.$
\end{lemma}
\begin{proof}
We follow the same argument as in Lemma~\ref{lem:R_tf: openess of quadratic tangency}.
    Suppose $\{f_t\}_t\in \mathcal E$. Let $t\in (0,1)$ be such that and $f_t$ has only (finitely many) quadratic periodic points, and has rotation number $p/q.$ We know the mode locking interval $I_f(p/q): = \rho\inv(p/q)$ is non-degenerate, and we may assume without loss of generality that $t = \sup I_f(p/q).$ Therefore, for any $\{g_t\}_t\in \Lambda_2$ such that $d_2(\{g_t\}-\{f_t\})$ is sufficiently small, there is some $\delta>0$ such that $t-\delta\in int(I_f(p/q)),$ and
    $$\rho(g_{t-\delta}) = \rho(f_{t-\delta}) = p/q.$$
    Since $\rho(f_{t+\delta})>\rho(f_t)$ for any $\delta>0,$ we can ensure that $$\rho(g_{t+\delta})>\rho(g_t) = p/q$$ provided $\{g_t\}_t$ is sufficiently close to $\{f_t\}_t.$ 
    Therefore, there exists $t'\in [t-\delta,t+\delta]$ such that $\rho(g_{t'}) = p/q$ and $g_{t'}$ has only parabolic periodic orbits. We just need to show that these are also quadratic. 

    By assumption, if $x\in S^1$ is a fixed point of $f_t^q$, then $|(f_t^q)''(y)|>0$ for any $y$ in a small neighborhood of $x.$ Since $g_{t'}^q$ and $f_t^q$ are $C^2-$close, all fixed points of $g_{t'}^q$ will be sufficiently close to those of $f_t^q,$ and so $|(g_{t'}^q)''|$ will be also bounded away from 0 at its fixed points. This completes the proof.
\end{proof}

    Combining Theorem~\ref{thm:Holder 1/2}, Lemma~\ref{lem:bottleneck}, Lemma~\ref{lem:denseness of tangency(general)}, and Lemma~\ref{lem:openess of quadratic tangency(general)}, we obtain a proof of Theorem~\ref{thm:generic C^2 1/2-Holder}.

\section{Dense Obstruction in $C^{1+\alpha}$ Families and Generic Loss of Hölder Continuity for $C^1$ Families}

This section will be dedicated to the proof of Theorem~\ref{thm:alpha/(1+alpha) Holder in C^(1+alpha) dense} and Theorem~\ref{thm:generic C^1 not Holder}. We consider the space $\Lambda_1$ of all monotone $C^1$ families $\{f_t\}_t:[0,1]\to \Diff_+^{1}(S^1)$ with $\rho(t_1)>\rho(t_0),$ equipped with the metric $d_1(\{f_t\}_t,\{g_t\}_t): = \|F(t,x)-G(t,x)\|_{C^1}$,
and the subset $\Lambda_{1+\alpha}\subseteq \Lambda_1$ consisting of all families in $\Lambda_1$ that are also $C^{1+\alpha}.$

Under these assumptions, we first construct a $d_1-$dense subset $\mathcal E'\subseteq \Lambda_{1+\alpha}$ such that for all $\{f_t\}_t\in \mathcal E',$ then rotation number of the family $\rho(t): = \rho(f_t)$ is at most $\frac{\alpha}{1+\alpha}-$Hölder. Then, we deduce from these facts that the set of families in $\Lambda_1$ for which the rotation number is not Hölder is dense $G_\delta.$

The idea behind the upper bound on the regularity of $\rho$ in this $C^{1+\alpha}$ case is identical to that of the $C^2$ case, which we discussed in detail in section 2. The key here is again Lemma~\ref{lem:bottleneck}, which allows us to estimate from above the number of steps needed for a point $x$ to clear the $C^{1+\alpha} $ bottleneck under iterations of $f_{t_0+\varepsilon}^q.$ More specifically, as long as our family is ``good'' in the sense that near each periodic point $x_k$ of $f_{t_0}$ we have
$$f_{t_0}^q(x)\geq x+c|x-x_k|^{1+\alpha},$$
then the rotation number cannot be $\beta-$Hölder for any $\beta>\alpha/(1+\alpha).$

Therefore, to complete the proof of Theorem~\ref{thm:alpha/(1+alpha) Holder in C^(1+alpha) dense}, we just need to show that the collection of ``good'' families in $\Lambda_{1+\alpha}$ is dense in the $d_1-$topology:
\begin{lemma}
    \label{lem:denseness of |x|^(1+alpha) in C^(1+alpha) families}
    Let $\mathcal E'\subseteq \Lambda_{1+\alpha}$ consist of strictly increasing families $\{f_t\}_t$ such that there exists some $t_0\in (0,1)$ for which $\rho(f_{t_0}) = p/q\in \Q$ with $\{x_k\}_{k = 0}^{q-1} = \{f_{t_0}^k(x_0)\}_{k = 0}^{q-1}$ as periodic points, and for each $x_k$ there is some constant $c>0$ with
    $$T(x)\geq x+c|x-x_k|^{1+\alpha}$$
    for any $x$ sufficiently close to $x_k.$ Here $T$ is the unique lift of $f_{t_0}^q$ with fixed points. Then $\mathcal E'$ is dense in $\Lambda_{1+\alpha}$ under the $d_1-$topology.
\end{lemma}
\begin{proof}
    Consider $\Lambda_2$ and $\mathcal E\subseteq \Lambda_2$ as defined in Theorem~\ref{thm:generic C^2 1/2-Holder} and Lemma~\ref{lem:denseness of tangency(general)}. Since $\Lambda_2$ is dense in $\Lambda_{1+\alpha}$ and $\mathcal E$ is dense in $\Lambda_2,$ we just need to show that any family $\{f_t\}_t\in\mathcal E$ can be perturbed into a family in $\mathcal E'.$ The specific construction in Lemma~\ref{lem:denseness of tangency(general)} tells us that we may assume that $\{x_k\}_{k = 0}^{q-1} = \{f_{t_0}^k(x_0)\}_{k = 0}^{q-1}$ are the only periodic points of $f_{t_0}.$ 
    
    Without loss of generality assume $x_0 = 0,$ and for ease of notation let us write $f: = f_{t_0}$ with lift $F.$

    Now take a small interval $V = B_{\delta}(x_{q-1})$ around $x_{q-1}$ such that $V$ contains no other periodic points of $f.$ Take another neighborhood $U = B_{\delta'}(x_{q-1})$ where $0<\delta'<\delta.$ Let $\varphi\in C^2(S^1)$ be any function that is equal to $f^{-(q-1)}$ on $U,$ and vanishes outside of $V.$ Let $\phi$ be the corresponding $1-$periodic function on $\R$ that vanishes outside of the lifts of $V.$ Define the circle map $g$ by the lift
    $$G: = F+\varepsilon|\phi|^{1+\alpha}$$ 
    for some sufficiently small $\varepsilon>0.$ Then clearly $\|G-F\|_{C^1}$ can be made as small as we wish. By construction $g^k(x_0) = f^k(x_0)$ for all $k\leq q-1,$ and $g^{q}(x_0) = g(f^{q-1}(x_0)) = x_0.$ And for small enough $\varepsilon$, $g$ will be a $C^{1+\alpha}$ diffeomorphism with exactly $\{x_k\}_k^{q-1}$ as periodic points.

    Let $S$ be the unique lift of $g^q$ with fixed points, and $T$ the unique lift of $f^q$ with fixed points. Now it remains to check that for any $x_k,$ 
    $$S(x)\geq x+c|x-x_k|^{1+\alpha}$$
    for all $x$ sufficiently close to $x_k.$ Indeed, if $x$ is close enough to $x_0,$ then $f^{q-1}(x)\in U$ and $f^k(x)$ will be disjoint from $V$ for all $k\leq q-2$.
    So we have
    \begin{align*}
        S(x)
        & = G(F^{q-1}(x))-p\\
        & = (F(F^{q-1}(x))-p)+\varepsilon|\phi(F^{q-1}(x))|^{1+\alpha}\\
        & = T(x)+\varepsilon|F^{-(q-1)}(F^{q-1}(x))|^{1+\alpha}\\
        & = T(x)+\varepsilon|x|^{1+\alpha}\\
        &\geq x+\varepsilon|x-x_0|^{1+\alpha}.
    \end{align*}
    But this estimate near $x_0$ can be transported to similar estimates near the other $x_k.$ Indeed, let us fix $1\leq k\le q-1$ and write $h: = f^k,$ where we have
    $$0<m\leq h'(x)\leq M$$
    for some constants $m$ and $M.$
    Note that any point close to $x_k = h(x_0) = h(0)$ has the form $h(x),$ where $x$ is close to $x_0.$ Then for such an $x$ we have $h(x) = f^k(x) = g^k(x).$ Writing $H = F^k,$ we have
    \begin{align*}
        S(H(x)) 
        &= H(S(x))\\
        &\geq H(x+\varepsilon|x|^{1+\alpha})\\
        &\geq H(x)+m\varepsilon|x|^{1+\alpha}\\
        &\geq H(x)+m\varepsilon M^{-(1+\alpha)}|H(x)-H(0)|^{1+\alpha},
    \end{align*}
    as desired.

Therefore, we may simply define our family $\{g_t\}_t$ by the lift $G_t: = F_t+\varepsilon|\phi|^{1+\alpha}$.
\end{proof}

    Theorem~\ref{thm:alpha/(1+alpha) Holder in C^(1+alpha) dense} follows immediately from Lemma~\ref{lem:denseness of |x|^(1+alpha) in C^(1+alpha) families} and Lemma~\ref{lem:bottleneck}.

Finally, we obtain Theorem~\ref{thm:generic C^1 not Holder} as a corollary:
\begin{corollary}
    \label{lem:not beta Holder dense G delta}
    For any $\beta\in (0,1),$ the collection of families in $\Lambda_1$ for which the rotation number is not $\beta-$Hölder is dense $G_\delta$ under the $d_1-$topology. As a consequence, the collection of families in $\Lambda_1$ for which the rotation number is not Hölder is dense $G_\delta.$
\end{corollary}
\begin{proof}
    Take any $\alpha\in (0,1)$ such that $\frac{\alpha}{1+\alpha}<\beta,$ and let $n\in \N.$ Consider those $\{f_t\}_t\in \Lambda_{1+\alpha}$ whose rotation number is not $(n,\beta)-$Hölder, i.e., there are $t_1,t_2\in [0,1]$ such that $|\rho(f_{t_1})-\rho(f_{t_2})|>n|t_1-t_2|^\beta.$ 
    
    Then by Theorem~\ref{thm:alpha/(1+alpha) Holder in C^(1+alpha) dense} and the fact that $\Lambda_{1+\alpha}$ is dense in $\Lambda_1,$ the set of families with rotation number not $(n,\beta)-$Hölder is dense in $\Lambda_1.$ Also notice that by continuity of the rotation number, the set of families in $\Lambda_1$ with rotation number not $(n,\beta)-$Hölder is also open.

    Intersecting over all $n\in \N$ shows that the families with rotation number not $\beta-$Hölder is dense $G_\delta.$ Again intersecting over all $m\in \N$ with $\beta = 1/m$ proves Theorem~\ref{thm:generic C^1 not Holder}.
\end{proof}

\section*{Acknowledgements}
I would like to express my gratitude to my advisor Anton Gorodetski, who attracted my attention to this problem, provided numerous references and suggestions, and carefully proofread the draft; I would also like to thank Victor Kleptsyn, for the helpful discussions and remarks; and finally, I would like to thank Íris Emilsdóttir, for her ideas in proving Lemma~\ref{lem:bottleneck}.

\bibliographystyle{amsplain}
\bibliography{references}

\end{document}